\documentclass[12pt, a4paper]{article}
\usepackage{amsthm, amssymb, mathrsfs, amsmath, color, enumerate}
\usepackage{easytable}
\usepackage{hyperref}
\usepackage{cleveref}
\hypersetup{
    colorlinks,
    citecolor=black,
    filecolor=black,
    linkcolor=black,
    urlcolor=black
}

\theoremstyle{plain}
\newtheorem{theorem}{Theorem}[section]
\newtheorem{lemma}[theorem]{Lemma}
\newtheorem{corollary}[theorem]{Corollary}
\newtheorem{proposition}[theorem]{Proposition}

\theoremstyle{definition}
\newtheorem{definition}[theorem]{Definition}

\newtheorem{remark}[theorem]{Remark}
\newtheorem{example}[theorem]{Example}

\DeclareMathOperator{\Out}{Out}
\DeclareMathOperator{\Inn}{Inn}
\DeclareMathOperator{\Aut}{Aut}
\DeclareMathOperator{\Dc}{Dc}
\DeclareMathOperator{\id}{id}

\newcommand{\bbz}{\mathbb{Z}}

\newcommand{\st}{\ |\ }

\title{On involutions and generalized symmetric spaces of dicyclic groups}
\author{Abigail Bishop, Christopher Cyr,\\ John Hutchens (\texttt{jdhutchens@saumag.edu}),\\ Clover May,\\ Nathaniel Schwartz (\texttt{nschwartz2@washcoll.edu})\\ and Bethany Turner}
\begin{document}
\maketitle

\begin{abstract}
Let $G=\Dc_{n}$ be the dicyclic group of order $4n$. Let $\varphi$ be an automorphism of $G$ of order $k$. We describe $\varphi$ and the generalized symmetric space $Q$ of $G$ associated with $\varphi$. When $\varphi$ is an involution, we describe its fixed point group $H=G^{\varphi}$ along with the $H$-orbits and $G$-orbits of $Q$ corresponding to  the action of $\varphi$-twisted conjugation.
\end{abstract}

%==============================================================================
%==============================================================================
\section*{Introduction}
%==============================================================================
%==============================================================================
Let $G$ be a group and $\varphi \in \Aut(G)$ such that $\varphi^k = \id$. The set $Q = \{x\varphi(x)^{-1} \st x \in G\}$ is known as the generalized symmetric space of $G$ corresponding to $\varphi$. If $\varphi$ is an involution and $G$ is a real reductive Lie group, then $Q$ is a reductive symmetric space. If $G$ is a reductive algebraic group defined over an algebraically closed field $F$, then $Q$ is also known as a symmetric variety; and if $G$ is defined over a non-algebraically closed field $F$, then $Q_F := \{x\varphi(x)^{-1} \st x \in G_F\}$ is called a symmetric $F$-variety, where $G_F$ denotes the $F$-rational points of $G$. Reductive symmetric spaces and symmetric $F$-varieties occur in geometry \cite{PdC83}, \cite{Abe88} and singularity theory \cite{LV83}; they are perhaps most well known for their role in the study of representation theory \cite{Vog83}. The generalized symmetric spaces defined above are also of importance in a number of areas, including group theory, number theory and representation theory \cite{BB81}.

To study the symmetric spaces, we begin by classifying automorphisms of $G$, particularly involutions, up to isomorphism. We also classify the $G$- and $H$-orbits of the action given by $\varphi$-twisted conjugation in $Q$, where $H = G^{\varphi} =  \{g \in G\st \varphi(g) = g\}$ is the fixed point group of $\varphi$.  For $g \in G$ and $q \in Q$, we define $\varphi$-twisted conjugation by  
$$g * q = gq\varphi(g)^{-1},\ g \in G, \ q \in Q,$$
and we denote the set of orbits in $Q$ corresponding to the $\varphi$-twisted action of $G$ or $H$ on $Q$ by $G\backslash Q$ or $H \backslash Q$. When $\varphi$ is an involution $Q \cong G/H$ via the map $g \mapsto g\varphi(g)^{-1}$, and also $H \backslash Q \cong H \backslash G/H$. These orbits and double cosets play an important role in representation theory \cite{He94}, \cite{He00}.

We begin by stating general definitions and facts about dicyclic groups. In Section 2 we describe the automorphism group of $G$ and the automorphisms of order $k$. We also determine the isomorphism classes of the involutions.  In Section 3 we describe $Q$ and $H$. The space $R = \{g \st \varphi(g) = g^{-1}\}$ is called the set of elements split by an automorphism $\varphi$.  This space arises in the study of related algebraic groups \cite{He00}, \cite{He88} and representation theory \cite{Vog83}. As was shown in \cite{Heletal12}, we prove that equivalent involutions do not necessarily have isomorphic fixed point groups, which is not the case when considering symmetric $F$-varieties of algebraic groups.

%==============================================================================
%==============================================================================
\section{Preliminaries}
%==============================================================================
%==============================================================================

A summary of the basic properties of dicyclic groups can be found in the literature; see for example \cite{HSMCox80}. Dicyclic groups can be described concisely in terms of their generators and relations. To construct the dicyclic group, fix an element $x$ of order $2n$ such that $\langle x \rangle = C_{2n}$, the cyclic group of order $2n$. Then adjoin an element $y$, subject to the conditions 
\[y^2 = x^n, \text{ and } y^{-1}xy = x^{-1}. \]

We shall use the following presentation for the dicyclic group throughout this paper:
$$G = \Dc_n = \langle x,y \ | \ x^{2n}=1, y^2=x^n, y^{-1} x y = x^{-1} \rangle, $$
when $n \ge 2$.  The group $\Dc_2$ is isomorphic to the quaternion group. The group $\Dc_1$ is isomorphic to $C_4$, which is not considered a dicyclic group.  In the literature, the dicyclic groups may be denoted by $\langle 2,2, n\rangle$ or $\langle n, 2, 2\rangle$, etc., as they belong to the family of binary polyhedral groups.

We begin by recording some basic facts about the structure and presentation of dicyclic groups.  

\begin{proposition}
\label{uniqpres}
Every element of $\Dc_n$ has a unique presentation as $y^a x^b$ where $a \in \{0,1\}$ and $0\leq b \leq 2n-1$.
\end{proposition}

\begin{proof}
Applying the relation $xy = yx^{-1}$ to an arbitrary element of the form $x^l y^k$ yields
$$
x^l y^k =
\begin{cases}
	y^k x^l \text{ , } k \text{ even} \\
	y^k x^{-l} \text{ , } k \text{ odd}.
\end{cases}
$$
Thus, given any product of elements $y^{l_1}x^{k_1}y^{l_2}x^{k_2} \ldots y^{l_m}x^{k_m}$, we can rearrange the terms to
$y^{l_1+l_2+\cdots +l_m} x^{\pm k_1 \pm k_2 \pm \cdots \pm k_m}$, and the result follows by applying the relations $y^2 = x^n$ and $x^{2n} = 1$ and by reducing the resulting exponent of $x$ modulo $2n$.
\end{proof}

\begin{corollary}
The dicyclic group has order $4n$ and consists of the elements 
\[ \Dc_n = \{1, x, \dots, x^{2n-1}, y, yx, \dots, yx^{2n-1}\}. \]
\end{corollary}

The following lemma describes the inverses of arbitrary elements of $Dc_n$.

\begin{lemma}
\label{inverse}
$\left( x^b \right)^{-1} = x^{2n-b}$ and $\left( yx^b \right)^{-1} = yx^{b+n}$.
\end{lemma}

\begin{proof}
	\[ x^b x^{2n-b} = x^{2n-b} x^b = x^{2n} =1 \]
	and
	\[ yx^b yx^{b+n} = yx^{b+n} yx^b = y^2 x^n = x^{2n} =1. \]
\end{proof}

\begin{theorem}
\label{normsubgrp}
The cyclic group $C_{2n} = \langle \ x \ | \ x^{2n}=1 \ \rangle$ is a normal subgroup of $\Dc_n$.
\end{theorem}

\begin{proof}
This set is clearly closed under multiplication and conjugation by elements of the subgroup. The set is also closed under conjugation by elements of the form $yx^b$, since 
$$yx^b x^k yx^{b+n} = y^2 x^{-b-k+b+n} = x^{2n-k} = x^{-k} \in C_{2n}.$$
\end{proof}

\begin{theorem}
\label{centergrp}
$Z(G) = \{ 1, x^n \}$.
\end{theorem}

\begin{proof}
The normal subgroup $C_{2n}$ is abelian, so it remains to check that $1$ and $x^n$ are the only elements that commute with elements of the form $yx^j$. We have
$$x^b y x^j  = y x^{j-b} = yx^j x^{-b},$$
and $x^b= x^{-b}$ if and only if $b=n$ or $b=0$. Since,
$$yx^k x^i = x^{-i} yx^k,$$
and $yx^k$ does not commute with all elements of $C_{2n}$.
\end{proof}

\begin{lemma}
\label{orderfour}
Elements of the form $y x^b$ have order $4$ in $G$. If $n$ is odd, these are the only order $4$ elements. If $n$ is even, $x^{\frac{n}{2}}$ and $x^{\frac{3n}{2}}$ are the only additional elements of order $4$.
\end{lemma}

\begin{proof}
Let $yx^b \in G$. Then $$\left( yx^b \right)^4 = yx^b yx^b yx^b yx^b = y^2 x^{b-b}y^2 x^{b-b} = y^4 =1,$$ and $\left( yx^b \right)^k \neq 1$ for all $1 \leq k \leq 3$. This calculation does not depend on $n$, and $(yx^b)^2 = x^n$ for all possible values of $b$. Now letting $(x^b)^4=1$ and solving $4b \equiv 2n$ for $b$ yields $b=0, \frac{n}{2}, n$, and $\frac{3n}{2}$. For $b = 0$ or $n$, we get $x^b = 1$ or $x^b = x^n$. These elements have order 1 and 2, respectively.
\end{proof}

%==============================================================================
%==============================================================================
\section{Automorphisms of $G$}
%==============================================================================
%==============================================================================

The automorphism group of $G$ is denoted by $\Aut(G)$. Denote by $\Inn(g)$ the inner automorphism defined by conjugation by $g \in G$; that is $x \mapsto gxg^{-1}$. Also, let $\Inn(G)$ denote set of inner automorphisms of $G$. The following are some general observations about $\Aut(G)$. We also characterize the automorphisms of fixed order and describe equivalence classes of automorphisms. We assume, unless otherwise specified, that $n>2$ and equivalence is modulo $2n$, denoted by $\equiv$.

\begin{proposition}\label{autgg}
The set $\big(\Aut(G),G\big)$ is a group with respect to the operation
\[ (\phi,g)(\theta,h) = \big(\phi \circ \theta, g \phi(h) \big), \]
where $\phi,\theta \in \Aut(G)$ and $g,h \in G$.
\end{proposition}

\begin{proof}
It is clear that the set is closed under the operation.  The element of the form $(\id, 1)$ is the identity element, since
\[ (\id,1)(\phi,g) = \big(\id \phi, 1 \id(g)\big) = (\phi, g) = \big(\phi \circ \id, g \phi(1) \big)= (\phi, g)(\id, 1). \]
An element of the form $(\phi, g)$ satisfies
\[ \big(\phi,g)(\phi^{-1},\phi^{-1}(g^{-1})\big) = \Big(\id, g \phi \big(\phi(g^{-1})\big)\Big) = (\id, 1), \]
and so every element has an inverse. Notice that, for $\rho \in \Aut(G)$ and $\ell \in G$,
\begin{align*}
\big( (\phi,g)(\theta,h) \big) (\rho, \ell) &= \big(\phi \circ \theta, g \phi(h)\big)( \rho, \ell) \\
&= \big( \phi \circ \theta \circ \rho, g \phi(h)\phi \circ \theta(\ell) \big),
\end{align*}
and that
\begin{align*}
(\phi,g)\big( (\theta,h)(\rho,\ell) \big) &= (\phi,g)\big(\theta  \circ \rho, h \theta(\ell) \big) \\
&= \Big(\phi \circ  \theta \circ  \rho, g \phi \big( h \theta(\ell) \big) \Big) \\
&= \big(\phi \circ  \theta \circ  \rho, g \phi(h) \phi \circ  \theta(\ell) \big).
\end{align*}
This shows the operation is associative and completes the proof.
\end{proof}

\begin{proposition}
\label{autfixC}
If $\varphi \in \Aut(G)$, then $\varphi(x) = x^r$ for $r \in \mathcal{U}_{2n}$.
\end{proposition}

\begin{proof}
	Since $n>2$, $x$ has order greater than $4$, and the only elements of order higher than $4$ are in $C_{2n}$.  So $\varphi(x)=x^j$ where $j$ and $2n$ are coprime.  If $\varphi(x^b) \not\in C_{2n}$, then $\varphi(x^b)= yx^k$, but $yx^k$ has order $4$, so $x^b$ must have order $4$.  The only elements of order $4$ in $C_{2n}$ are $x^{\frac{n}{2}}$ and its $x^{\frac{3n}{2}}$ when $n$ is even.  So we let $n$ be even; then
		\[ \varphi(x^{\frac{n}{2}})=\varphi(x)^{\frac{n}{2}} = x^{\frac{jn}{2}} \in C_{2n}. \]
\end{proof}

We can now classify the automorphisms of $G$ by specifying their actions on the generators $x$ and $y$.  Denote by $\mathcal{U}_{2n}$ the set of units in $\mathbb{Z}_{2n}$.

\begin{theorem}
\label{automorphisms}
A homomorphism $\varphi : G \to G$ is an automorphism if and only if $\varphi(x) = x^r$ and $\varphi(y)=yx^s$, for some $r\in \mathcal{U}_{2n}$ and some $s\in \mathbb{Z}_{2n}$.
\end{theorem}

\begin{proof}
Let $\varphi$ be an arbitrary automorphism of $G$. By \cref{autfixC}, $\varphi(x)=x^b$ for some $b \in \mathcal{U}_{2n}$, and since $\varphi$ preserves order, $\varphi(x)=x^r$ for some $r\in \mathcal{U}_{2n}$. Since $\varphi$ is injective, $\varphi(y)=yx^s$ for some $s\in \mathbb{Z}_{2n}$.
\end{proof}

\begin{corollary}
$\Aut(Dc_n)$ is a group of the form $\big(\Aut(G),G\big)$ as described in \cref{autgg}.
\end{corollary}

We can view elements of $\Aut(G)$ as ordered pairs in $\mathcal{U}_{2n} \times \bbz_{2n}$. Then the group operation in the first coordinate is multiplication within $\mathcal{U}_{2n}$, and the operation in the second coordinate is automorphism-twisted addition in $\bbz_{2n}$ since the composition of two automorphisms $\varphi_{(r,s)} \circ \varphi_{(p,q)}$ takes the form $(r,s)(p,q) = (rp, s + rq)$. Here multiplication by an element of $\mathcal{U}_{2n}$ in $\bbz_{2n}$ is an automorphism of $\bbz_{2n}$.  The identity element is $(1,0)$, and the inverse of an element $(r,s)$ is $(r,s)^{-1} = (r^{-1}, -r^{-1}s)$. 

Based on \cref{automorphisms}, every automorphism of $\Dc_n$ can be written uniquely as $\varphi_{(r,s)}$, where $\varphi_{(r,s)}(x) = x^r$ and $\varphi_{(r,s)}(y) = yx^s$. When there is no confusion, we may omit the subscript. 

\begin{proposition}\label{innfixC}
An automorphism $\varphi$ of $G$ is inner if and only if $s$ is even and either $r = 1$ or $2n-1$.
\end{proposition}

\begin{proof}
Consider $\Inn(g)$ for some $g \in G$. Either $g = x^b$ or $g = yx^b$, where $b \in \bbz_{2n}$. Suppose $g = x^b$. Then
	\[ \Inn(x^b)(x) = (x^b) x (x^{-b}) = x, \]
	and
	\[ \Inn(x^b)(y) = (x^b) y (x^{-b}) = yx^{-2b}. \]
Next, let $g=yx^b$. Then
	\[ \Inn(yx^b)(x) = (y x^b) x (yx^{b+n}) = y^2 x^{-b-1+b+n} = x^n x^{n-1} = x^{2n-1}, \]
	and
	\[ \Inn(yx^b)(y) = (yx^b) y (yx^{b+n}) = y x^b y^2 x^{b+n} = y x^{b+n+b+n} = yx^{2b}. \] \\
Conversely, let $r = 1 \mbox{ and } s$ be even. Note that conjugation by $x^{-s/2}$ gives 
$$x^{-s/2}xx^{s/2} = x$$
and $$x^{-s/2}yx^{s/2} = yx^{s}.$$
Thus $\varphi \in \Inn(G)$.  Similarly, if $r = 2n-1 \mbox{ and } s$ is even, then conjugation by $yx^{s/2}$ gives 
$$(yx^{s/2})x(yx^{s/2+n}) = y^2x^{-s/2-1}x^{s/2+n} = x^{n-1+n} = x^{2n-1}$$
and 
$$(yx^{s/2})y(yx^{s/2+n}) = yx^{s}.$$
Again, $\varphi \in \Inn(G)$.
\end{proof}

\begin{remark}\label{outer}
Note that if $\varphi$ is an outer automorphism, then either $s$ is odd, or else $r \notin \{1,2n-1\}$.
\end{remark}

\begin{lemma}\label{char_order_m}
Let $\varphi \in \Aut(G)$. Then $\varphi^{k}=\id$ if and only if $r^k \equiv 1$ and $s(1+r+r^2+...+r^{k-1}) \equiv 0$.
\end{lemma}

\begin{proof}
From the definition of $\varphi$ we obtain 
$$\varphi^{k}(x)= \varphi^{k-1}(x^r) = \varphi^{k-1}(x)^r = \varphi^{k-2}(x^r)^r = \varphi^{k-2}(x^{r^2}) = \cdots = x^{r^{k}}$$
and 
\begin{align*}
\varphi^{k}(y) &= \varphi^{k-1}(yx^s)\\
&= \varphi^{k-2}(yx^sx^{rs})\\
&= \varphi^{k-2}\left(yx^{s(1+r)}\right)\\
& \hspace{1cm}\vdots\\
&= yx^{s(1+r+r^2+\dots+r^{k-1})}.
\end{align*}
Then $\varphi^{k}=\id$ if and only if it fixes the generators $x$ and $y$.
\end{proof}

We can also count the solutions to the equation $r^2 \equiv 1$ under certain conditions of $n$ using the following result from \cite{OmOmOu09}.

\begin{lemma}
\label{sqrootunity}
Let $S_2(2n)$ denote the solutions to $r^2 \equiv 1 \mod 2n$.  Then
	\[
	|S_2(2n) | =
		\begin{cases}
		2^{\omega(k)}, n=k \text{ where } k \text{ is odd } \\
		2^{\omega(k)+1}, n=2k \text{ where } k \text{ is odd } \\
		2^{\omega(k)+2}, n=2^{\alpha}k \text{ where } k \text{ is odd and } \alpha> 1,
		\end{cases}
	\]
and $\omega(k)$ is the number of distinct prime factors of $k$.
\end{lemma}

\begin{proof}
See \cite{OmOmOu09}.
\end{proof}

\Cref{char_order_m} leads to the following simple characterization of the involutions of $G$.

\begin{corollary}
\label{char involution}
Let $\varphi \in \Aut(G)$. Then $\varphi$ is an involution if and only if $r^2 \equiv 1$ and $s(1+r) \equiv 0$.
\end{corollary}

When $\varphi$ is inner, the involutions of $G$ are succinctly characterized as follows.
\begin{lemma}\label{inninv}
$\Inn(g)$ is an involution if and only if $g$ has order $4$.
\end{lemma}

\begin{proof}
$\Inn(g)$ is an involution if and only if $g^2\ell g^{-2} = \ell$ for any $\ell \in G$, so $g^2 \in Z(G) = \{1, x^n\}$. If $g^2 = 1$ then $g = x^n$ and $\Inn(x^n) = \id$. If $g^2 = x^n$, then $g$ has order 4. Assume $g$ has order 4. Then by \cref{orderfour}, $g = yx^b, x^{\frac{n}{2}}$ or $x^{\frac{3n}{2}}$. Notice that $(yx^b)^2 = (x^{\frac{n}{2}})^2  = (x^{\frac{3n}{2}})^2 = x^n \in Z(G)$.
\end{proof}

\begin{definition}
\label{isomorphydef}
Two automorphisms $\varphi$ and $\vartheta$ are said to be \emph{isomorphic}, and we write $\varphi \sim \vartheta$, if there exists $\sigma \in \Aut(G)$ such that $\sigma\varphi = \vartheta \sigma$. Two isomorphic automorphisms are said to be in the same \emph{isomorphy class}.
\end{definition}

\begin{proposition}
\label{isomorphy}
Two automorphisms $\varphi_{(r,s)}$ and $\varphi_{(p,q)}$ are isomorphic if and only if $r = p$ and $qu - s$ is a multiple of $r - 1$ for some $u \in \mathcal{U}_{2n}$.
\end{proposition}

\begin{proof}
By \cref{isomorphydef}, $\varphi_{(r,s)} \sim \varphi_{(p,q)}$ implies there is some $\sigma = \varphi_{(u,v)} \in \Aut(G)$ such that $\sigma\varphi_{(p,q)} = \varphi_{(r,s)} \sigma$. Direct computation shows that $\varphi_{(r,s)} \circ \varphi_{(u,v)} = \varphi_{(ru, s+rv)}$ and $\varphi_{(u,v)} \circ \varphi_{(p,q)} = \varphi_{(pu, v+qu)}$. These equations hold if and only if $ru \equiv pu$ and $s+rv \equiv v+qu$. Since $r$ and $p$ are both in $\bbz_{2n}$, the first equation simplifies to $r \equiv p$ which implies that $r = p$. Similarly, since $u \in \mathcal{U}_{2n}$, the second equation reduces to $qu - s \equiv rv - v = v(r-1)$.
\end{proof}
 
\begin{example}
\label{isomorphy_1n}
Consider the isomorphy class of $\varphi_{(1,n)}$. \Cref{isomorphy} implies that the  possible automorphisms in this class are of the form $\varphi_{(1,q)}$, where $q$ satisfies  $qu - n \equiv 0$. Hence $n \equiv qu$ for some $u \in \mathcal{U}_{2n}$.
\end{example}

\begin{example}
Let $n = 3$, and consider $\varphi_{(5,2)}$ and $\varphi_{(5,0)}$. From the proof of \cref{isomorphy}, it follows that $\varphi_{(5,1)} \circ \varphi_{(5,2)} =\varphi_{(25,2+5(1))}=\varphi_{(1,1)}$ and $\varphi_{(5,0)} \circ \varphi_{(5,1)} =\varphi_{(25,1+5(0))}=\varphi_{(1,1)}$,which implies that $\varphi_{(5,2)} \sim \varphi_{(5,0)}$.  
\end{example}

The following proposition characterizes the inner automorphisms $\varphi$ of order $k > 2$.

\begin{proposition}\label{inn_inv_val}
Suppose $n>2$ and $\Inn(g)$ has order $k$, where $g \in G$.  The inner automorphisms are one of the following types
\begin{enumerate}[(a)]
\item $g^k=1$, $g=x^b$, and $bk \equiv 0 \pmod{2n}$,
\item $g^k=x^n$, $g=x^b$, and $bk \equiv n \pmod{2n}$
\item $g=yx^b$ and $\Inn(g)$ is an involution.
\end{enumerate}
\end{proposition}

\begin{proof}
For the case $g=yx^b$ we know that $g^4=1$ and so $\Inn(g)$ is an inner involution by \cref{char involution} and \cref{orderfour}.  Let $\Inn(g)^k(z) = \id$, i.e., $g^kzg^{-k} = z$ for all $z \in G$. Thus $g^k \in Z(G) = \{1, x^n\}$, by \cref{centergrp}.  Let $g^k = 1$, and suppose $g = x^b$. Then $x^{bk} = 1$ and hence $bk \equiv 0$. Next, let $g^k = x^n$, and suppose $g = x^b$. Then $g^k = x^{bk} = x^n$, and hence $bk \equiv n$. 
\end{proof}

%==============================================================================
\subsection{Automorphisms of $\Dc_{2}$}
%==============================================================================
We consider the dicyclic group $\Dc_2$ as a special case, since $\Dc_2$ is isomorphic to the quaternion group $Q$. Moreover, since all non-central elements have order 4, \cref{autfixC} does not hold, and there are additional automorphisms. The dicyclic group $\Dc_2$ is presented as  
$$\Dc_2 = \{1, x, x^2, x^3, y, yx, yx^2, yx^3\}$$
subject to the relations 
$$x^4 = y^4 = 1,\ \ \ \  x^2 = y^2,\ \ \ \  xy = yx^{-1}.$$
The multiplication table of $\Dc_2$ is given in \cref{multdic2}.

\begin{table}[h!]
\caption{Multiplication in $\Dc_2$}
\begin{center}\label{multdic2}%+
\begin{TAB}(e,1cm,1cm){c|cccc:cccc}{c|cccc:cccc}
& $1$ & $x$ & $x^2$ & $x^3$ & $y$ & $yx$ & $yx^2$ & $yx^3$\\
$1$ & $1$ & $x$ & $x^2$ & $x^3$ & $y$ & $yx$ & $yx^2$ & $yx^3$\\
$x$ & $x$ & $x^2$ & $x^3$ & $1$ & $yx^3$ & $y$ & $yx$ & $yx^2$\\
$x^2$ & $x^2$ & $x^3$ & $1$ & $x$ & $yx^2$ & $yx^3$ & $y$ & $yx$\\
$x^3$ & $x^3$ & $1$ & $x$ & $x^2$ & $yx$ & $yx^2$ & $yx^3$ & $y$\\
$y$ & $y$ & $yx$ & $yx^2$ & $yx^3$ & $x^2$ & $x^3$ & $1$ & $x$\\
$yx$ & $yx$ & $yx^2$ & $yx^3$ & $y$ & $x$ & $x^2$ & $x^3$ & $1$\\
$yx^2$ & $yx^2$ & $yx^3$ & $y$ & $yx$ & $1$ & $x$ & $x^2$ & $x^3$\\
$yx^3$ & $yx^3$ & $y$ & $yx$ & $yx^2$ & $x^3$ & $1$ & $x$ & $x^2$
\end{TAB}
\end{center}
\end{table}

Since homomorphisms preserve order and $x^2$ is the only element of order 2, every automorphism of $\Dc_2$ must fix both 1 and $x^2$.  Moreover, every automorphism is uniquely determined by its action on the generators $x$ and $y$.  We determine the inner automorphisms by conjugating the generators by each element of $\Dc_2$.

For all $r,s \in \{1, 2, \dots, n\}$, conjugation by $x^r$ and $x^s$ is equivalent. Thus there is exactly one inner automorphism corresponding to conjugation by a power of $x$. We  denote it by $\Inn(x)$. Conjugation by $yx^2$ and conjugation by $y$ are equivalent, as are conjugation by $yx^3$ and $yx$. These last two complete the list of inner automorphisms. We list all inner automorphisms in \cref{inndic2}.

\begin{table}[h!]
\caption{Inner Automorphisms of $\Dc_2$}
\begin{center}
\begin{tabular}{lllc}
Name & $x\mapsto$ & $y\mapsto$ & Order\\
\hline
$\id$ & $x$ & $y$ & 1\\
$\Inn(x)$ & $x$ & $yx^2$ & 2\\
$\Inn(y)$ & $x^3$ & $y$ & 2\\
$\Inn(yx)$ & $x^3$ & $yx^2$ & 2\\
\end{tabular}
\end{center}
\label{inndic2}%+
\end{table}

Since each non-trivial inner automorphism has order 2, it follows that the inner automorphism group is isomorphic to $V$, the Klein-4 group. Moreover, the outer automorphism group is isomorphic to $S_3$, and the entire automorphism group is isomorphic to $S_4$. These isomorphisms are described below.  The complete listing of automorphisms of $\Dc_2$ is given in \cref{autdic2}.

\begin{table}[h]
\caption{Representative Outer Automorphisms of $\Dc_2$}
\begin{center}
\begin{tabular}{lllc}
Name & $x\mapsto$ & $y\mapsto$ & Order\\
\hline
$\id$  & $x$ & $y$ & 1\\
$\varphi_1$ & $x^3$ & $yx^3$ & 2\\ 
$\varphi_2$ & $yx^2$ & $x^3$ & 2 \\
$\varphi_3$ & $yx^3$ & $yx^2$ & 2\\
$\varphi_4$ & $y$ & $yx$ & 3\\
$\varphi_5$ & $yx$ & $x$ & 3\\
\end{tabular}
\end{center}
\label{outdic2}%+
\end{table}

Using the representatives in \cref{outdic2}, we exhibit an isomorphism of $\Out(\Dc_2)$ and $S_3$ by the map given in \cref{outdic2s3}.
\begin{table}[h]
\caption{Isomorphism of $\Out(\Dc_2)$ and $S_3$}
\begin{center}
\begin{tabular}{lcccccc}
$\id$ & $\varphi_1$ & $\varphi_2$ & $\varphi_3$ & $\varphi_4$ & $\varphi_5$\\
$\updownarrow$& $\updownarrow$& $\updownarrow$& $\updownarrow$& $\updownarrow$& $\updownarrow$\\
$\id$ & $(12)$ & $(13)$ & $(23)$ & $(123)$ & $(132)$
\end{tabular}
\end{center}
\label{outdic2s3}
\end{table}

By this map, the cosets of $S_4/V$ as elements of $S_3$ are exactly the cosets of the outer automorphism group $\Aut(\Dc_2)/\Inn(\Dc_2)$. 

%==============================================================================
%==============================================================================
\section{Symmetric spaces of $G$}
%==============================================================================
%==============================================================================
In this section we describe the symmetric space $Q$ and the fixed-point group $H=G^{\varphi}$, for a given automorphism of order $k$. When $\varphi$ has order $2$ it follows that $Q \cong G/H$. Recall that 
$$H = \{g \in G\st \varphi(g) = g\}$$ and 
$$Q = \{x\varphi(x)^{-1} \st x \in G\}.$$

\begin{proposition}\label{q_and_h}
Let $\varphi_{(r,s)}$ be an automorphism of order $k$. Then
$$Q=\{x^{b(1-r)} \text{ and } x^{s+b(r-1)} \st b \in \mathbb{Z}_{2n}\},$$
and
$$H=\{x^{b} \st b(1-r) \equiv 0 \} \cup \{yx^{b} \st b(1-r) \equiv s \}.$$
\end{proposition}

\begin{proof}
The elements of $G$ also in $Q$ are those such that:
\begin{align}
\label{Q1}x^b\varphi(x^b)^{-1} &= x^b(x^{rb})^{-1} = x^{b(1-r)}\\
\label{Q2}yx^b\varphi(yx^b)^{-1} &= yx^b(yx^{s + rb})^{-1} = yx^byx^{s+rb+n} = y^2x^{s-b+rb+n} = x^{s+b(r-1)}
\end{align}
Similarly, the elements of $H$ are those such that:
\begin{align}
\label{H1}\varphi(x^b) &= x^{rb} = x^b,\\
\label{H2}\varphi(yx^b) &= yx^{s+rb} = yx^b,
\end{align}
It follows that $x^b \in H$ if $b(1 - r) \equiv 0$ and $yx^b \in H$ if $s + rb \equiv b$.
\end{proof}

\begin{example}
Consider $\varphi_{(5,2)}$ with $n=3$. By \cref{q_and_h}, $Q_{(5,2)}=\{1, x^2, x^4\}$ and $H_{(5,2)}=\{1,x^3,yx, yx^4\}$.
\end{example}

The descriptions of $Q$ and $H$ are more specific when $\varphi$ is an inner automorphism. Using \cref{innfixC}, if $\varphi$ is inner, $s$ is even and $r = 1$ or $2n-1$.

\begin{corollary}\label{innerqh}
Let $\varphi$ be a non-trivial inner automorphism. Then
\begin{enumerate}[$(1)$]
\item $Q=\{1,x^s\}$ and $H=\{1,x,x^2, \ldots, x^{2n-1}\}$ when $r=1$, and
\item $Q=\{1,x^2,\ldots, x^{2n-2} \}$ and $H=\{1,x^n, yx^{\frac{s}{2}} , yx^{{\frac{s}{2}}+n} \}$ when $r=2n-1$.
\end{enumerate}
\end{corollary}

\begin{proof}
Begin with $r = 1$.  By \cref{q_and_h}, $Q = \{1, x^s\}$. Since $b(1-r) = (0) \equiv 0$ for every $b \in \mathbb{Z}_{2n}$, it follows that $H \supseteq C_{2n}$. By \cref{H2}, $b \equiv s + b$ implies that $s \equiv 0$, and $s \neq 0$ since $\varphi \neq \id$. So $H = C_{2n}$.

Now let $r = 2n - 1$. Again, by \cref{q_and_h}, $Q = \{x^{2b}, x^{s-2b}\}$, but since $s$ is even, we get only the even powers of $x$.  For $H$, equation \cref{H1} implies that $2b \equiv 0$, which has solutions $b = 0$ and $n$. Thus, $1$ and $x^n$ are in $H$. \Cref{H2} implies that $b \equiv s - b$ or $2b \equiv s$.  Since $\frac{s}{2}$ solves $2b \equiv s$, and $2(\frac{s}{2}+n) = s + 2n$ implies that $s + 2n \equiv s$, $yx^{s/2}$ and $yx^{s/2+n}$ are in $H$.  
\end{proof}

\begin{remark}
If $r=1$ and $s=n$ is even, then $\varphi$ is an involution.  Also, if $r=2n-1$, $\varphi$ is an involution for any choice of $s \in \mathbb{Z}_{2n}$.
\end{remark}

\begin{remark}
	There exist non-isomorphic automorphisms with identical fixed point groups. For example, let $n=4$ and consider $\varphi_{(3,0)}$ and $\varphi_{(7,0)}$. By \cref{q_and_h}, $H_{(3,0)}=H_{(7,0)}=\{x^b \st b(1 - 3) \equiv 0 \} \cup \{yx^b \st b \equiv 0 + 3b\}$. That is, $2b \equiv 0$, and hence $b = 4$. Thus $H = \{1,x^{4}, y, yx^{4}\}$. Since $3 \neq 7$, by \cref{isomorphy} $\varphi_{(3,0)} \not \sim \varphi_{(7,0)}$.
\end{remark}

%==============================================================================
\subsection{Elements split by an automorphism}
%==============================================================================

We say that an element $g \in G$ is \emph{split} by an automorphism if $\varphi(g) = g^{-1}$. We denote by $R$ the set of all elements split by $\varphi$, i.e.,
$$R = R_{\varphi} = \{g \in G \ | \ \varphi(g) = g^{-1}\}.$$
 
\begin{proposition}
\begin{displaymath}
R_{\varphi_{(r,s)}}=\{x^b \st b(r+1) \equiv 0\} \cup \{yx^b \st s+rb \equiv n+b\}.
\end{displaymath}
\end{proposition}

\begin{proof}
We examine the action of $\varphi$ on the two types of elements in $G$:
$$\varphi(x^b) = x^{rb} = x^{-b} \Leftrightarrow rb \equiv -b \Leftrightarrow b(r+1) \equiv 0,$$
$$\varphi(yx^b) = yx^{s+rb} = yx^{b+n} \Leftrightarrow s + rb \equiv b + n.$$
\end{proof}

\begin{corollary}
\label{RminusQ}
\begin{displaymath}
R - Q = \{yx^l \st s + rl \equiv n + l\} \cup \{x^k \st k(r + 1) \equiv 0, k \notin\langle 1-r \rangle, k \notin s + \langle r - 1\rangle \}
\end{displaymath}
\end{corollary}

\begin{example}
\label{varphi1n}
Consider $\varphi_{(1,n)}$. Since $\varphi|_{C_{2n}} = \id$, the only elements split by $\varphi$ are those of order two, namely 1 and $x^n$.  Also, $\varphi(yx^b) = yx^{n+b} = (yx^b)^{-1}$, so $yx^{b}\in R$ for all $b$.  Thus, $R = \{1, x^n, y, yx, ..., yx^{2n-1}\}$.  From \cref{q_and_h}, $Q = \{1, x^n\}$, so $R - Q = \{y, yx, ..., yx^{2n-1}\}$.
\end{example}

\subsection{The action of $G$ on $Q$}

The group $G$ acts on $Q$ by $\varphi$-twisted conjugation, which we denote by $*$.  For $\varphi \in \Aut(G)$, we define the action of $G$ on $Q$ by
$$g * q = g q \varphi(g)^{-1},$$
where $g\in G$ and $q \in Q$.  Given $q = h\varphi(h)^{-1} \in Q$, this becomes $$g \ast q = g \ast h \varphi(h)^{-1} \mapsto gh\varphi(h)^{-1}\varphi(h)^{-1} = gh\varphi(gh)^{-1}.$$ We denote the orbits in $Q$ with respect to $\varphi$-twisted conjugation ($*$) by $G\backslash Q$.

\begin{proposition}
\label{prop:G-orbits}
Let $\varphi \in \text{Aut}(G)$.
\begin{enumerate}[$(1)$]
\item If $s\notin \langle 1-r\rangle$, then the $H$-orbits in $Q$ are:
\begin{displaymath}
H\backslash Q = \big\{  \{x^{j} \} \st j\in \langle 1-r\rangle \cup (s+ \langle r-1\rangle)\big\} = \big\{ \{ x^{j} \} \st x^{j}\in Q \big\}.
\end{displaymath}
\item
If $s\in \langle 1-r\rangle$, then the $H$-orbits in Q are:
\begin{displaymath}
H\backslash Q = \big\{  \{x^{j},x^{-j} \} \st j\in \langle 1-r\rangle \cup (s+ \langle r-1\rangle)\big\} = \big\{ \{ x^{j},x^{-j} \} \st x^{j}\in Q \big\}.
\end{displaymath} 
\end{enumerate}
In each case, there is a single $G$-orbit in $Q$, so $G\backslash Q = \{ Q\}$.
\end{proposition}
 \begin{proof}
Let $\varphi \in \text{Aut(G)}$. Because $H$ is the set of fixed points of $\varphi$, the action of $H$ on $Q$ is simply conjugation. If $s\notin \langle 1-r\rangle$, $H\subset \langle x\rangle$ by \cref{q_and_h}. So for $h_{1}=x^{i} \in H, q=x^{j} \in Q$, $h_{1} \ast q = x^{i} x^{j} x^{-i}=x^{j}=q$, and thus $H$ fixes $Q$ pointwise. If $s\in \langle 1-r\rangle$, then $\{yx^{b} : b\equiv s+rb\} \subset H$ so we must consider the action of these elements on $Q$ as well. Fixing $h_{2}=yx^{k} \in H$, we see that $h_{2} \ast q = yx^{k}x^{j}yx^{k+n} = y^{2}x^{-k-j}x^{k+n}=x^{n}x^{-k-j}x^{k+n}=x^{-j}$. Then $h_{2}$ sends $x^{-j}$ to $x^{j}$, so the orbit of $x^{j}$ under $H$ is $\{x^{j},x^{-j}\}$.

Now we show that there is exactly one $G$-orbit in $Q$. Every element of $Q$ is of the form $q_1=x^{b(1-r)}$ or $q_2=x^{s+b(r-1)}$. Notice $x^{-b}\in G$ with $x^{-b}\varphi(x^{-b})^{-1} = x^{-b}(x^{-rb})^{-1} = x^{-b}x^{2n-rb}=x^{b(1-r)}=q_{1}$, and $yx^{b}\varphi(yx^{b})^{-1} = yx^{b}(yx^{s+rb})^{-1} = yx^{b}yx^{n+s+rb}=y^{2}x^{n+s+b(r-1)}=x^{s+b(r-1)}=q_{2}$. Therefore $G\backslash Q = \{ Q\}$.
\end{proof}

In the following section we provide a summary of results for the special case when $n=2$.

%==============================================================================
\subsection{Symmetric Spaces of $\Dc_2$}
%==============================================================================

We list the fixed-point groups $H$ and symmetric spaces $Q$ for all automorphisms in \cref{dic2appendix}. For each involution, we also compute the set $R$ of $\varphi$-split elements. Finally, we determine the $H$-orbits in $Q$ corresponding to each automorphism. 

For the identity automorphism, $H = G$ implies that the $H$-orbits and $G$-orbits are identical, and $Q = \{1\}$ implies they are both $\{1\}$. 

For the remaining inner automorphisms, $Q = Z(G) = \{1, x^2\}$, so $g * q = gq\varphi(g)^{-1} = qg\varphi(g)^{-1}$. In particular, for $h \in H$, $h * q = qh\varphi(h)^{-1} = qhh^{-1} = q$, so the $H$-orbits are $\{q\}$ for all $q \in Q$. In fact, we get the same result if $H_{\varphi} = Z(G)$, since $h * q = hq\varphi(h)^{-1} = hqh^{-1} = qhh^{-1} = q$.  So if $Q_{\varphi} = Z(G)$ or $H_{\varphi} = Z(G)$, the $H$-orbits in $Q_{\varphi}$ are of the form $\{q\}$ for all $q \in Q$. 

By the above analysis, we have determined the $H$-orbits for every automorphism except the last six listed. In what follows, we use the fact that orbits partition a set. For each of these automorphisms, $Q = \{1, q\}$ for some nontrivial $q \in G$.  Since $h * 1 = h1\varphi(h)^{-1} = hh^{-1} = 1$, $O_1 = \{1\}$.  Thus $O_q = \{1\}$ or $O_q = \{q\}$, and hence $O_q = \{q\}$. Thus for all $\varphi \in \Aut(\Dc_2)$, the $H$-orbits are of the form $O_q = \{q\}$ for some $q \in Q$.

Every automorphism fixes $1$ and $x^2$, so $\{1, x^2\} \subseteq H_{\varphi}$ for all $\varphi \in \Aut(G)$.  Since $H$ is non-empty, $1 = h\varphi(h)^{-1} = hh^{-1} \in Q_{\varphi}$ for all $h \in H, \varphi \in \Aut(G)$.  But then calculating the orbit of $1$ gives $O_1 = \{a \ast 1 = a1\varphi(a)^{-1} = a\varphi(a)^{-1} | \hspace{2pt} a \in G\} = Q_{\varphi}$.  Since $q \in O_q$ for all $q \in Q$ and $O_1 = Q$ implies that $O_q = Q$ for all $q \in Q$.  This analysis holds for every automorphism, so there is a single $G$-orbit for every automorphism. In summary, we have the following proposition.

\begin{proposition}
Let $G = \Dc_2$, $\varphi \in \Aut(G)$, and $H$ and $Q$ be the fixed point group and symmetric space corresponding to $\varphi$. Under $\varphi$-twisted conjugation, $G$ acts transitively on $Q$, and the $H$ orbits in $Q$ are always single elements.
\end{proposition}

\bibliographystyle{plain}
\bibliography{gssdihedral}

\appendix
\section{$\Dc_2$}\label{dic2appendix}

\begin{table}[h]
\caption{Symmetric Spaces and Fixed-Point Groups for $\Dc_2$}
\begin{center}
\begin{tabular}{ccccc}\label{autdic2}
Automorphism & Order & $H$ & $Q$ & $R$\\
\hline
$\id$ & 1 & $\Dc_2$ & $\{1\}$ & \\
$\Inn(x)$ & 2 & $\{1, x, x^2, x^3\}$ & $\{1, x^2\}$ & $\{1, x^2, y, yx, yx^2, yx^3\}$\\
$\Inn(y)$ & 2 &$\{1, x^2, y, yx^2\}$ & $\{1, x^2\}$ & $\{1, x, x^2, x^3, yx, yx^3\}$\\
$\Inn(yx)$ & 2 & $\{1, x^2, yx, yx^3\}$ & $\{1, x^2\}$ & $\{1, x, x^2, x^3, y, yx^2\}$\\
$\varphi_1$ & 2 & $\{1, x^2\}$ & $\{1, x, x^2, x^3\}$ & $\{1, x, x^2, x^3\}$\\
$\varphi_2$ & 2 &$\{1, x^2\}$ & $\{1, x^2, yx, yx^3\}$ & $\{1, x^2, yx, yx^3\}$\\
$\varphi_3$ & 2 & $\{1, x^2\}$ & $\{1, x^2, y, yx^2\}$ & $\{1, x^2, y, yx^2\}$\\
$\varphi_1 \circ \Inn(x)$ & 2 & $\{1, x^2\}$ & $\{1, x, x^2, x^3\}$ & $\{1, x, x^2, x^3\}$\\
$\varphi_2 \circ \Inn(yx)$ & 2 & $\{1, x^2\}$ & $\{1, x^2, yx, yx^3\}$ & $\{1, x^2, yx, yx^3\}$\\
$\varphi_3 \circ \Inn(y)$ & 2 & $\{1, x^2\}$ & $\{1, x^2, y, yx^2\}$ & $\{1, x^2, y, yx^2\}$\\
$\varphi_4$ & 3 &$\{1, x^2\}$ & $\{1, x, y, yx\}$\\
$\varphi_5$ & 3 &$\{1, x^2\}$ & $\{1, x^3, yx^2, yx^3\}$\\
$\varphi_4 \circ \Inn(x)$ & 3 & $\{1, x^2\}$ & $\{1, x^3, yx, yx^2\}$\\
$\varphi_4 \circ \Inn(y)$ & 3 & $\{1, x^2\}$ & $\{1, x, yx^2, yx^3\}$\\
$\varphi_4 \circ \Inn(yx)$ & 3 & $\{1, x^2\}$ & $\{1, x^3, y, yx^3\}$\\
$\varphi_5 \circ \Inn(x)$ & 3 & $\{1, x^2\}$ & $\{1, x, yx, yx^2\}$\\
$\varphi_5 \circ \Inn(y)$ & 3 & $\{1, x^2\}$ & $\{1, x, y, yx^3\}$\\
$\varphi_5 \circ \Inn(yx)$ & 3 & $\{1, x^2\}$ & $\{1, x^3, y, yx\}$\\
$\varphi_1 \circ \Inn(y)$ & 4 & $\{1, x, x^2, x^3\}$ & $\{1, x^3\}$\\
$\varphi_1 \circ \Inn(yx)$ & 4 & $\{1, x, x^2, x^3\}$ & $\{1, x\}$\\
$\varphi_2 \circ \Inn(x)$ & 4 & $\{1, x^2, yx, yx^3\}$ & $\{1, yx^3\}$\\
$\varphi_2 \circ \Inn(y)$ & 4 & $\{1, x^2, yx, yx^3\}$ & $\{1, yx\}$\\
$\varphi_3 \circ \Inn(x)$ & 4 & $\{1, x^2, y, yx^2\}$ & $\{1, y\}$\\
$\varphi_3 \circ \Inn(yx)$ & 4 & $\{1, x^2, y, yx^2\}$ & $\{1, yx^2\}$\\
\end{tabular}
\end{center}
\end{table}

\end{document}